\documentclass[twoside,a4paper, 10pt]{amsart} 

\usepackage{amsmath, amsthm, amssymb, amsxtra,amsfonts, color,calligra,mathrsfs,comment,url, accents} 
\usepackage{enumerate,verbatim,mathrsfs,comment}
\usepackage[all,2cell,ps]{xy}
\usepackage{mathtools} 
\usepackage{extpfeil}   
\usepackage{soul}  
\usepackage{tikz,tikz-cd,enumerate}
\usetikzlibrary{matrix,arrows,decorations.pathmorphing}
\usepackage{emptypage} 
\usepackage{floatrow} 
\usepackage{lineno} 
\usepackage{float}
\usepackage{xcolor} 
\usepackage{graphicx}
\usepackage{tikz,tikz-cd,tkz-graph,enumerate}
\usepackage{tabularx}
\usetikzlibrary{positioning}
\usepackage[pagebackref]{hyperref}

   \usepackage{setspace}
   
   \usepackage[a4paper, total={6in, 8in}]{geometry}
\DeclareMathOperator{\Hom}{Hom}

\DeclareMathOperator{\rank}{rank}
\DeclareMathOperator{\conv}{conv}
\DeclareMathOperator{\diag}{diag}

\DeclareMathOperator{\Max}{Max}   

\DeclareMathOperator{\intt}{int}

\newcommand{\Um}{\mbox{Um}}

\newcommand{\GL}{\mbox{GL}}
\newcommand{\KK}{\mbox{$K_1$}}

\newcommand{\El}{\mbox{E}}
\newcommand{\Sl}{\mbox{SL}}
\newcommand{\Ep}{\mbox{ESp}}

\DeclareMathOperator{\Etrans}{ETrans}
\DeclareMathOperator{\Sp}{Sp}

\newcommand{\m}{\mathfrak m}

\newcommand{\sm}{\mbox{$M_*$}}

\newcommand{\pz}[1]{{\mathbb{Z}}_{+}^{#1}}
\newcommand{\nz}[1]{{\mathbb{Z}}^{#1}}

\newcommand{\lrinn}[2]{\langle {#1}\,, {#2} \rangle}
\newcommand{\Lrinn}[2]{\Big\langle {#1}\,, {#2} \Big\rangle}

\def\ESp{\rm ESp}

\def\Sp{\rm Sp}
\def\vp{\rm \vspace{0.2cm}}
\def\Um{\rm Um}
\def\I{\rm Id}
\def\KK{\rm K}
\def\Sll{\rm SL}
\def\ETranss{\rm ETrans}

\theoremstyle{plain}
\newtheorem{theorem}{Theorem}[section]

\newtheorem{thm}{Theorem}
\newtheorem{proposition}[theorem]{Proposition}
\newtheorem{corollary}[theorem]{Corollary}
\newtheorem{lemma}[theorem]{Lemma}

\theoremstyle{definition}
\newtheorem{definition}[theorem]{Definition}

\newtheorem{example}[theorem]{Example}

\theoremstyle{remark}
\newtheorem{remark}[theorem]{Remark}
\numberwithin{equation}{section}
\AfterEndEnvironment{figure}{\noindent\ignorespaces} 

\begin{document}
	
		\title{${\KK}_1$-Stability of symplectic modules over monoid algebras}
	
	\author{Rabeya Basu}
\address{Indian Institute of
	Science Education and Research (IISER) Pune,  India} 
\email{rabeya.basu@gmail.com, rbasu@iiserpune.ac.in}
\thanks{$^*$Corresponding Author}

\author{Maria A. Mathew*
}
\address{Indian Institute of
	Science Education and Research (IISER) Pune,  India} 
\email{maria.mathew@acads.iiserpune.ac.in}
\thanks{Research by the second author was supported by the NBHM post-doctoral research grant.}

	\date{}
	\maketitle

	\noindent
	{\small Abstract: Let $R$ be a regular ring of dimension $d$ and $L$ be a $c$-divisible monoid. If ${\KK}_1{\Sp}(R)$ is trivial and $k \geq d+2,$ then we prove that the symplectic group ${\Sp}_{2k}(R[L])$ is generated by elementary symplectic matrices over $R[L]$. When $d \leq 1$ or $R$ is a geometrically regular ring containing a field, then improved bounds have been established. We also discuss the linear case, extending the work of \cite{Gubeladze-ClassicalMR1079964}. } \vp



	\section{Introduction}\label{intro}
	The enquiry into ${\KK}_1$-stability of rings is a natural consequence of the results that have been obtained for its well known counterpart ${\KK}_0$. 
	To be precise, when $A=k[X_1, \ldots, X_n],$ where $k$ is a field, the affirmative to Serre's problem by Quillen-Suslin, gave the best non-stable version for ${\KK}_0(A)$. This led to questions on the possibility of the best stable and non-stable version for the classical groups ${\KK}_1(A)$ and ${\KK}_1{\Sp}(A)$. In \cite{Suslin-MR0472792} and \cite{Kop-MR497932}, Suslin and Kopeiko established these versions by giving stabilization theorems for ${\KK}_1(R[X_1, \ldots, X_m, Y_1^{\pm{1}}, \ldots, Y_{n}^{\pm{1}}]),$ where $R$ is a Noetherian ring.
	
	The subsequent development of $K$-theoretic techniques leading upto the proof of Serre's problem inspired similar developments for monoid algebras, the starting point being Anderson's conjecture \cite{MR526663-Anderson} (proved by J. Gubeladze in \cite{Gubeladze2-maximalMR937805}). A setback to the expectation of ${\KK}_1$-analogues for monoid algebras, was due to an example by Srinivas \cite{Srinivasc-ceg-MR0918839}. Building on Srinivas's work, Gubeladze \cite{Gubeladze-NontrivMR1360174} showed that for any regular ring $R,$ the expected ${\KK}_1$ analogue for the monoid algebra $R[U, UV, UV^2]$ cannot be attained. This example also works to show the same for the symplectic group  ${\Sp}(R[U, UV, UV^2])$.

	Despite the counterexamples, Gubeladze \cite{Gubeladze-ClassicalMR1079964} found a class of monoids algebras that satisfy the expected ${\KK}_1$-analogues. It was shown, when $R$ is a Euclidean domain and $L$ a $c$-divisible algebra ($c>1$) (for definition see Section \ref{s2}), one gets ${\Sl}_{n}(R[L]) = {\El}_{n}(R[L])$ for $n \geq 3$. These $c$-divisible monoids arise naturally in tandem with any finite abelian group. Given any finite abelian group $G,$ one can show there exists a $c$-divisible  monoid $L$ ($c>1$) such that $G \simeq Cl(L),$ where $Cl(L)$ denotes the divisor class group of $L$ (see \cite{Chouinard-Div-MR0645239} for definition). Motivated by these results, we identify a class of rings $\mathcal{R}_t$ (see Section \ref{s3}) in the symplectic setup and prove:
	
	\begin{thm}\label{t1}
		Let  $R \in \mathcal{R}_k$ and $L$ be a c-divisible monoid, where $c > 1$ and $k \geq 1$. Further, if we have ${\Sp}_{2k}(R[X_1, \ldots, X_m, Y_1^{\pm{1}}, \ldots, Y_n^{\pm{1}}])= {\ESp}_{2k}R[X_1, \ldots, X_m, Y_1^{\pm{1}}, \ldots, Y_n^{\pm{1}}],$ then
		$${\ESp}_{2k}(R[L]) = {\Sp}_{2k}(R[L]). $$
	\end{thm}
	
	We show that any regular local ring of dimension $d$ belongs to $\mathcal{R}_{k},$ for all $k \geq d+2$. As a consequence we are able to prove that:
	
	\begin{thm}\label{t6}
		Let $R$ be a regular ring of dimension d with ${\KK}_1{\Sp}(R)=0$ and $L$ be a $c$-divisible monoid. Then for any symplectic module $P$ of rank $2k$ with $k \geq \max\{2,d+1\},$
		$${\Sp}(P \perp R[L]^2) =  {\Etrans}_{\Sp}(P \perp R[L]^2).$$
	\end{thm}
	
	This result aligns with the best possible answer for the general case in the setup of monoid algebras. When $d \leq 1$ or $R$ is geometrically regular ring containing a field, further improvements can be made. When $M$ is free monoid, $i.e., R[M] \simeq R[X_1, \ldots, X_m, Y_1^{\pm{1}}, \ldots, Y_n^{\pm{1}}],$ the result for unstable ${\KK}_1(R[M])$ is due to Suslin ($cf.$ \cite{Suslin-MR0472792}, Corollary 7.10). Using the techniques of this paper, it can also be shown that for a regular ring $R$ with trivial ${\KK}_1,$ we gain the equality ${\Sl}_{n}(R[L]) = {\El}_{n}(R[L])$ for $n \geq \max\{3,d+2\},$ thus extending the work started in \cite{Gubeladze-ClassicalMR1079964}.
	
	Another sought after ${\KK}_1$-analogue is courtesy of Grothendieck's work on ${\KK}_0$-homotopy invariance. When $R$ is regular, Grothendieck proved that ${\KK}_0(R) \simeq {\KK}_0(R[X_1, \ldots, X_m, Y_1^{\pm{1}}, \ldots, Y_n^{\pm{1}}])$. In \cite{BassHellerSwan-MR174605}, Bass-Heller-Swan established the ${\KK}_1$-analogue for polynomial rings. Gubeladze found an element in ${\KK}_1(R[U,UV,UV^2]) \smallsetminus {\KK}_1(R)$ for any regular ring $R$. He observed that this failure of ${\KK}_1(R[M])$-analogue was a consequence of the (absence of) excision property for ${\KK}_1(R[M]),$ where $M$ is a monoid.  When $R$ is regular and $M$ an affine simplicial monoid, it was observed \cite{Gubeladze-NontrivMR1360174} that the only monoid algebras satisfying ${\KK}_1(R) \simeq {\KK}_1(R[M])$ are polynomial algebras ($i.e.,$ when $M$ is free). Fortunately, when $L$ is a $c$-divisible monoid, Gubeladze \cite{Gub-geo&al-MR1414169} proved the presence of excision in $K$-theory of $R[L]$. In (\cite{Gubeladze-ClassicalMR1079964}, Theorem 2.24) the expected ${\KK}_1(R[L])$-invariance was established. We establish ${\KK}_1{\Sp}(R[L])$-invariance when $L$ is $c$-divisible:
	\begin{thm}\label{t2}
		Let R be a regular ring and L a $c$-divisible monoid, where $c > 1$. Then 
		$${\KK}_1{\Sp}(R[L]) \simeq {\KK}_1{\Sp}(R).$$
	\end{thm}
	
	A question on the existence of a non-free, affine monoid which is ${\KK}_1$-invariant for all regular rings, was raised in \cite{Gubeladze-NontrivMR1360174}. To answer this question, Krishna-Sarwar \cite{SarwarKrishna-MR3995720}, showed that if $R$ is a regular ring containing $\mathbb{Q},$ then ${\KK}_1(R) \simeq {\KK}_1(R[X,Y,ZX,ZY]).$ This work was extended by  Schappi (\cite{schäppi2023symplecticktheoryproblemmurthy}, Theorem 1.2), answering  this question in totality. There is a good chance that the same deductions can be made for ${\KK}_1{\Sp}$-invariance for the monoid algebra $R[X,Y,ZX,ZY]$. Further, the techniques developed here can be utilized to collect the ${\KK}_2{\Sp}$-invariance result of algebra generated by simplicial $c$-divisible monoids as in (\cite{Gubeladze-ClassicalMR1079964}, Theorem 3.1 and 3.4). It would also be interesting to see extension of these results for other suitable classical groups and their generalizations. We leave it to the reader to verify that the results presented here also come through for Laurent polynomial extension of these algebras as well, $i.e.$, for conditions provided in the hypothesis, ${\Sp}_{2k}(R[L][X_1, \ldots, X_m, Y_1^{\pm{1}}, \ldots, Y_n^{\pm{1}}])$ coincides with its elementary symplectic subgroup.
	
	\section{Prologue: Geometry of polarized monoids}\label{s2}
	
	All our rings will be commutative Noetherian with unity and monoids will be commutative cancellative and torsion-free. Let $R$ be a ring and $n \in \mathbb{N}$. Define the matrices over $R$ by:
	$$ \psi_1 = \begin{pmatrix}
		0 & 1 \\
		-1 & 0 
	\end{pmatrix} \text{ and } \psi_n = \psi_{n-1} \perp \psi_1$$
	Define $\widetilde{u}=u^T\psi_n$ for any $u \in R^{2n}$. The symplectic group ${\Sp}_{2n}(R)$ over a ring $R,$ is given by:
	$${\Sp}_{2n}(R)= \{ \alpha \in {\GL}_{2n}(R) \mid \alpha^T\psi_n\alpha = \psi_n \}$$
	Let $e_i$'s  denote the standard basis vectors with $1$ in the $i$'th position and $0$ elsewhere, and $e_{ij}$'s the matrix $e_ie_j^T$. Let ${\rm Id}_{2n}$ denote the identity matrix of size $2n$ and $\sigma$ be the product of transpositions, $$\sigma = (1,2)(3,4)\ldots(2n-1,2n).$$ 
	For $\lambda \in R$ and $1 \leq i < j \leq k=2n$ with $\sigma(i) \neq j,$ one defines 
	$$se_{ij}(\lambda) = {\rm Id}_{2n}+\lambda e_{ij} - \lambda (-1)^{i+j}e_{\sigma(j)\sigma(i)}$$ 
	For values of $n,i,j, \lambda$ as indicated above, the elementary symplectic group ${\ESp}_{2n}(R),$ is the subgroup of ${\Sp}_{2n}(R)$ generated by $e_{i(\sigma(i))}(\lambda)$'s and $se_{ij}(\lambda)$'s.

	Let $(M,+)$ be a commutative monoid. Then $M$ is cancellative, if for $x, y, z \in M,$ $x+y=x+z,$ then $y=z,$ and torsion free if for some $n \in  \pz{},$ $nx=0,$ then $x = 0$. The group completion of $M$ is denoted by $gp(M)$. The rank of a monoid $M,$ denoted by $\rank(M),$ is defined as the dimension of the $\mathbb{Q}$-vector space $\mathbb{Q} \otimes_{\nz{}} gp(M)$. A monoid is positive if its group of units is trivial. The monoid $M$ is normal, if for $x \in gp(M)$ and $n \in \pz{},$ $nx \in M,$ then $x \in M$. Let  $c \in \pz{}$. The monoid $M$ is called c-divisible, if for all $m \in M$ there exists $n \in M,$ such that $cn=m$.

	Since $M$ is cancellative and torsion-free, therefore $M$ can be embedded in $\mathbb{R}^{r},$ where $r=\rank(M)$. Henceforth, we may assume without loss of generality that $M \subseteq \mathbb{R}^{r},$ where $0$ coincides with the origin. For $x \in \mathbb{R}^r,$ we denote by $\phi(x),$ the intersection of the unit sphere $S^{r-1}$ (in $\mathbb{R}^r$) with the line joining the origin to $x$. By $\phi(M)$ we denote the smallest subset in $S^{r-1}$ containing the set $\{\phi(m) \mid m \in M \smallsetminus \{0\} \}$ such that if $x, y$ are two non-antipodal points in $\phi(M),$ then the shortest line joining $x$ to $y$ on $S^{r-1}$ is contained in $\phi(M)$. Here $\phi(M)$ is the convex hull of the set $\{\phi(m) \mid m \in M \smallsetminus \{0\} \}$ in $S^{r-1},$ written as: 
	$$\phi(M)=\conv(\{\phi(m) \mid m \in M \smallsetminus \{0\} \}).$$ 
	Given a convex (on $S^{r-1}$) subset $X \subseteq \phi(M),$ a new submonoid of $M$ can be defined by:
	$$M(X) = \{m \in M \smallsetminus \{0\} \mid \phi(m) \in X\} \cup \{0\}. $$
	
	\noindent Further denote, ${\intt}(M):=M(\intt(\phi(M)))$ and $\sm:= \intt(M) \cup \{0\}$. A point $P \in S^{r-1}$ is called rational, if there exists a rational point $x \in R^r$ such that $\phi(x) = P$. A convex subset $X$ of $S^{r-1}$ is defined to be rational, if it is the convex hull of rational points. A convex set $X$ (in $S^{r-1}$) is open, if there exists a subspace $H \subset\mathbb{R}^r$ such that $X$ is open in $H \cap S^{r-1}$.
	
	\begin{remark}\label{r1}
		(\cite{Gubeladze-ClassicalMR1079964}, \S I) Every open convex $(d$ -$1)$-dimensional subset in $S^{d-1}$ is rational and can be represented as union of $(d$ -$1)$ dimensional rational closed polytopes. Thus if $M$ is a $c$-divisible monoid with $\phi(M)$ open, then $M$ can be seen as direct limit of $c$-divisible monoids $M_i$'s with $\phi(M_i)$ as closed polytope.
	\end{remark}

	The following definition and notations are as per \cite{Gubeladze-ClassicalMR1079964}:
	
	\begin{definition}
		A triple $(P, \Gamma, M)$ is called a polarized monoid if
		\begin{enumerate}[(i)]
			\item $M$ is an affine normal monoid, $P$ is a rational point and $\Gamma$ a closed convex rational polytope such that $P$ is not in the affine space generated by any arbitrary proper face of $\Gamma$.
			\item $\dim(\Gamma) = \dim(\phi(M))$ and $\phi(M) = \conv(P, \Gamma)$.
			\item For any $(\dim(\Gamma) -1)$-dimensional face $\gamma$ of $\Gamma,$ the $M$-submonoid $M_{\gamma},$ defined as
			$M_{\gamma} := M(\conv(P, \gamma)),$ is generated by $ M({P}) \cup M(\gamma).$
		\end{enumerate}
		
	\end{definition}
	\noindent For the polarized monoid $(P, \Gamma, M),$ we fix the following notation:
	\begin{enumerate}[(i)]
		\item $P_-$ is the point on $S^{r-1}$ diametrically opposite to $P$.
		\item $\m$ is a maximal ideal of $R[M(\Gamma)]$ containing $M(\Gamma) \setminus \{0\}$. 
		\item $\mathcal{N}$ is the $M$-submonoid, of rank 1, given by $M(\{P\})$ and further $\mathcal{N} = \langle t \rangle$.
		\item $M_{-}$ denotes the submonoid of $K(M)$ generated by $M(\Gamma) \cup \{-m \mid m \in M(\{P\})\}$.
		
	\end{enumerate}
	
	\noindent Note that here $(P_-, \Gamma, M_-)$ also forms a polarized triple. For simplicity sake, instead of saying $(P, \Gamma, M)$ forms a polarized triple, we will say $M$ is a polarized monoid. We fix the notations, $$A=R[M(\Gamma)]_{\m}, B=R[M]_{\m}, B_{-}=R[M_{-}]_\m \text{ and } C=R[\mathcal{N}^{-1}M]_\m.$$ It can be shown that $B \cap B_- = A$ (Step 5 of Proposition 2.14 in \cite{Gubeladze-ClassicalMR1079964}).

	\section{Equality of Symplectic and Elementary symplectic groups}\label{s3}

	It is known that if $R$ is a local ring or a euclidean domain, then ${\Sp}_{2k}(R)={\ESp}_{2k}(R)$ for all $k \geq 1$. Kopeiko, in \cite{Kop-MR497932}, showed that $R$ is a field and $A=R[X_1, \ldots, X_m],$ then for $k \geq 2,$ the symplectic group ${\Sp}_{2k}(A)={\ESp}_{2k}(A)$. When $k=1,$ this equality doesn't hold due to \cite{Cohn-MR207856}. This equality is what we refer to as the unstable ${\KK}_1{\Sp}$. Clearly, unstable ${\KK}_1{\Sp}$ implies the results for stable ${\KK}_1{\Sp},$ $i.e., {\Sp}(A) = {\ESp(A)}$. The primary aim of this section is to determine base rings $R$ such that stable and unstable ${\KK}_1{\Sp}$ results holds monoid algebra generated by $c$-divisible monoids $L$.
	\newline
	
	Let $k \in \mathbb{N}$ and $R$ be a ring.  Let $R(X)$ denote the localization of the ring $R[X]$ with respect to all polynomials that have invertible leading coefficients in $R[X]$. Define $\mathcal{R}_k$ to be the class of rings satisfying the following: $R \in \mathcal{R}_k,$ if
	\begin{enumerate}[(i)]
		\item ${\ESp}_{2k}(R) = {\Sp}_{2k}(R)$ and
		\item if there exist finite rings  $R=R_1, R_2 \ldots, R_t=R'$ for $t>1,$ such that 
		\begin{enumerate}[(a)]
			\item $R_2=R(X)$ or $R_2=R_{\m}$ for some $\m \in \Max(R)$,
			\item  for $2 < i \leq t$ we have 
			$$ 
			R_i= \left\{
			\begin{array}{ll}
				R_{i-1}(X) & \text{if } R_{i-1}={(R_{i-2})}_{\m_{i-2}}, \\
				{(R_{i-1})}_{\m_{i-1}} & {\rm else}.\\
			\end{array} 
			\right. 
			,$$
			where $\m_{j} \in \Max(R_j)$ for $1 \leq j \leq t$.
		\end{enumerate}
		Then  ${\ESp}_{2k}(R') = {\Sp}_{2k}(R')$.
	\end{enumerate}
	
	Let $R \in \mathcal{R}_{t}$, for some $t$ such that $2t \geq \dim(R).$ By observing that $\dim(R_i) \leq \dim(R)$ for $R_i$'s as in $(ii)$, the surjective  part of the Stabilization Theorem in \cite{Vaserstein-MR0269722} gives us $R \in \mathcal{R}_{k}$ for all $k \geq t$.
	\begin{example}\label{r3}
		Let $(R, \m)$ be a regular local ring of dimension $\leq 1$. We claim that $R \in \mathcal{R}_k$ for all $k \geq 1$. From (\cite{Lam-MR2235330}, Chapter IV, Proposition 1.4), one may write $R(X) = R'_Y$, where $R'=R[Y]_{(\m,Y)}$ and $Y=X^{-1} \in R(X)$. Using (\cite{Lam-MR2235330}, Chapter IV, Corollary 6.2), we get $R(X)$ is a special PID and therefore
		$${\Sp}_2(R(X)) = {\Sl}_2(R(X))={\El}_2(R(X))={\Ep}_2(R(X)).$$
		
		By the above observation, we get $R \in \mathcal{R}_k$ for all $k \geq 1$. By (\cite{Vaser-Men-Fritz-MR1086811}), Theorem 1.8), (\cite{Kop-MR497932}, Lemma 4.1) and induction on $n,$  we can show that for  $k \geq 2$ and $m,n \geq 0,$ we have
		$${\Sp}_{2k}(R[X_1, \ldots, X_m, Y_1^{\pm{1}}, \ldots, Y_n^{\pm{1}}]) = {{\ESp}_{2k}}(R[X_1, \ldots, X_,, Y_1^{\pm{1}}, \ldots, Y_n^{\pm{1}}]).$$		
	\end{example}
	
	Due to \cite{Cohn-MR207856}, $k=2$ is the best bound we achieve in the above equation. Since it was shown that:
	$${\Sp}_{2}(\nz{}[X_1, \ldots, X_m]) \neq {{\ESp}_{2}}(\nz{}[X_1, \ldots, X_m]).$$
	
	The following theorem can be seen as the symplectic analogue of Theorem 1, 2 and 6 of \cite{BassHellerSwan-MR174605}. The proof comes through by using the localization exact sequence or simply following the workings of the proofs in the aforementioned theorems on replacing ${\GL}_n$ and ${\El}_n$ by ${\Sp}_{2n}$ and $\Ep_{2n},$ respectively.
	
	\begin{theorem}\label{ktheorysp}
		Let $R$ be a regular ring and $X$ an indeterminate. Then
		\begin{enumerate}[$($i$)$]
			\item ${\KK}_0{\Sp}(R) \simeq {\KK}_0Sp(R[X]) \simeq {\KK}_0{\Sp}(R[X,X^{-1}])$,
			\item ${\KK}_1{\Sp}(R[X]) \simeq {\KK}_1{\Sp}(R)$,
			\item ${\KK}_1{\Sp}(R[X,X^{-1}]) \simeq {\KK}_1{\Sp}(R)$.
		\end{enumerate}
	\end{theorem}
	
	\begin{corollary}\label{cktheorysp}
		Let $R$ be a regular ring of dimension $d$ with ${\KK}_1{\Sp}(R)=0$. For $k \geq d+2$ we have
		$${\ESp}_{2k}(R[X_1, \ldots, X_m, Y_1^{\pm{1}}, \ldots, Y_{n}^{\pm{1}}]) = {\Sp}_{2k}(R[X_1, \ldots, X_m, Y_1^{\pm{1}}, \ldots, Y_{n}^{\pm{1}}]).$$
	\end{corollary}

	\begin{example}\label{r5}
		Let $(R, \m)$ be a regular local ring of dimension $d$. We claim that $R \in \mathcal{R}_k$ for $k \geq d+2$. From (\cite{Lam-MR2235330}, Chapter IV, Proposition 1.4), one may write $R(X) = R'_Y$, where $R'=R[Y]_{(\m,Y)}$ and $Y=X^{-1} \in R(X)$. Since ${\KK}_0Sp$ and ${\KK}_1{\Sp}$ are trivial for local rings, on using the localization exact sequence one gets
		$$0 \simeq {\KK}_1{\Sp}(R') \rightarrow {\KK}_1{\Sp}(R'_Y) \simeq {\KK}_1{\Sp}(R(X)) \xrightarrow{\delta} {\KK}_0Sp(R'/YR') \simeq {\KK}_0{\Sp}(R) \rightarrow {\KK}_0{\Sp}(R'),$$
		we may conclude ${\KK}_1{\Sp}(R(X))=0$. As $\dim(R(X))=\dim(R) =d,$ by on using injective part of the Stabilization theorem in \cite{Vaserstein-MR0269722} we get ${\ESp}_{2k}(R(X)) = {\Sp}_{2k}(R(X)),$ for $k \geq d+2$. Compare with Example \ref{r3} to note that for $d=1$ we have a much improved result.
	\end{example}
	
	\begin{example}\label{r6}
		Let $(R, \m)$ be a geometrically regular ring containing a field. We may assume $d>1$. Then ${\KK}_1{\Sp}(R(X))=0$. By (\cite{Basuinjective-MR2578583}, Theorem 2), it follows ${\ESp}_{2k}(R(X)) = {\Sp}_{2k}(R(X))$ for $2k \geq \max\{3,d+1\}$. Thus $R \in \mathcal{R}_{\lceil k/2 \rceil}$. If $R$ is also local, by (\cite{Basuinjective-MR2578583}, Remark 3.16), ${\Sp}_{2k}(R[X]) = {\ESp}_{2k}(R[X])$ for $k \geq 2$. By inducting on $m+n$ and using (\cite{Kop-MR497932}, Lemma 4.1), for $k \geq 2$ we arrive at:
		$${\Sp}_{2k}(R[X_1, \ldots, X_m, Y_1^{\pm{1}}, \ldots, Y_n^{\pm{1}}]) = {{\ESp}_{2k}}(R[X_1, \ldots, X_,, Y_1^{\pm{1}}, \ldots, Y_n^{\pm{1}}]).$$
		This is a significant reduction compared to Example \ref{r5}.
	\end{example}
	
	We recall the following notations from Section \ref{s2}:
	$$A=R[M(\Gamma)]_{\m}, B=R[M]_{\m}, B_{-}=R[M_{-}]_\m, C=R[\mathcal{N}^{-1}M]_\m \text{ and } \mathcal{N}=\langle t \rangle. $$
	
	\noindent We now begin our work on the understanding the structure of ${\ESp}_{2n}(C)$:
	
	\begin{definition}
		For each $n \geq 1,$ we can define:
		$$\mathbb{M}^{(n)} = \{I \subseteq \{1, \ldots,2n\} \mid \forall~i \in I, \text{either } i \in I, \sigma(i) \notin I \text{ or } i \notin I, \sigma(i) \in I \}.$$
		Corresponding to $I \in \mathbb{M}^{(n)},$ a new diagonal matrix can be defined as:
		$$\delta_I = \diag(d_1, \ldots, d_{2n}),$$
		where
		$$ 
		d_i= \left\{
		\begin{array}{ll}
			t & i \in I, \\
			1 & {\rm else}.\\
		\end{array} 
		\right. 
		$$
	\end{definition}
	
	\begin{lemma}\label{l1}
		Let R be a ring and $M$ be a polarized monoid. Let $n \geq 2$ and $\alpha \in {\ESp}_{2n}(B) \cap {\Sp}_{2n}(B, tB)$. For $I \in \mathbb{M}^{(n)},$
		$$\delta_{I}\alpha{\delta_{I}}^{-1} \text{ and } {~\delta_{I}}^{-1}\alpha{\delta_{I}} \in {\ESp}_{2n}(B).$$
	\end{lemma}
	
	\begin{proof}
		The proof will focus on showing ${\delta_{I}}\alpha{\delta_{I}}^{-1} \in \Ep_{2n}(B)$ (the other part will follow using $t{\delta_{I}}^{-1}=\delta_{\sigma(I)},$ where $\sigma(I) = I^c \cap \{1, \ldots, 2n\} \in \mathbb{M}^{(n)}$). 
		We use the property of $M$ being polarized to write the decomposition:
		$$ \alpha = \prod\limits_{k=1}^m se_{{i_k}{j_k}}(a_k+tf_k), $$
		where $a_k \in A$ and $f_k \in B$. For $1 \leq p \leq m,$ define $\gamma_p =  \prod\limits_{k=1}^p se_{{i_k}{j_k}}(a_k) \in \Ep_{2n}(A).$
		This helps us rewrite $\alpha$ as
		$$\alpha = \Bigg( \prod\limits_{k=1}^m \gamma_k se_{{i_k}{j_k}}(tf_k) \gamma_k^{-1}\Bigg) \gamma_m $$
		Thus to show $\delta_{I}\alpha{\delta_I}^{-1} \in {\Sp}_{2n}(B),$ it is sufficient to show that 
		$$ \delta_{I}\big(\gamma_kse_{{i_k}{j_k}}(tf_k)\gamma_k^{-1}\big){\delta_I}^{-1}\text{ and } \delta_{I}\gamma_m{\delta_I}^{-1}  \in \Ep_{2n}(B).$$ 
		
		We break this proof into two parts and forgo writing the index $k$ for the sake of simplicity. Let 
		$$\alpha_1 = \delta_{I}\big(\gamma se_{{i}{j}}(tf)\gamma^{-1}\big){\delta_I}^{-1}$$
		and $v$ and $w$ be the $i$'th  and $\sigma(j)$'th column of $\gamma$. Define $v_1,w_1 \in (A)^{2n}$ as $v_1= \delta_Iv$ and  $w_1= \delta_Iw$. Let $(-1)^{\sigma(j)}f = f'$. If we assume $\sigma(j)=i,$ then have $\gamma se_{{i}{j}}(f)\gamma^{-1} = {\I}_{2n} +(-1)^{i}tfv\widetilde{v} =  {\I}_{2n} +tf'v\widetilde{v}$ and	
		\begin{align*}
			\hspace{5cm}	\alpha_1 &= \delta_I({\I}_{2n} + {tf'}v\widetilde{v}){\delta_{I}}^{-1}\\
			&= {{\I}}_{2n} + f'\delta_Iv{v}^T\psi_n({t\delta_{I}^{-1}})\\
			&= {\I}_{2n} + f'\delta_Iv{v}^T(\psi_n{\delta_{\sigma(I)}})\\
			&= {\I}_{2n} + f'v_1{v}^T\delta_I\psi_n\\
			&= {\I}_{2n} + f'v_1\widetilde{v_1}.
		\end{align*}
		This gives $\alpha_1 \in \Ep_{2n}(A)$ by (\cite{Kop-MR497932}, Lemma 1.5). Else $\gamma se_{{i}{j}}(f')\gamma^{-1} = {\I}_{2n} +tf'(v\widetilde{w}+w\widetilde{v})$. Then
		\begin{align*}
			\hspace{5cm}	\alpha_1 &= \delta_I({\I}_{2n} + {tf'}(v\widetilde{w}+w\widetilde{v})){\delta_{I}}^{-1}\\
			&= {\I}_{2n}+{f'}(\delta_Iv{w}^T+\delta_Iw{v}^T)\psi_n{\delta_{\sigma(I)}}\\
			&= {\I}_{2n} + {f'}(v_1{w}^T+w_1v^T)\delta_I\psi_n\\
			&= {\I}_{2n} + {f'}(v_1\widetilde{w}_1+w_1\widetilde{v_1}).
		\end{align*}
		\noindent Since $\gamma \in  \Ep_{2n}(A),$ therefore $v$  being a column of $\gamma$ implies $v \in {\Um}_{2n}(A)$. Then \begin{align*}
			\hspace{5cm}(f'\widetilde{w})\cdot v &= f'w^T\psi_n\gamma e_i\\
			&= f'(\gamma^{-1}w)^T\psi_ne_i\\
			&=f'e_{\sigma(j)}^T\psi_ne_i\\
			&=0.
		\end{align*}
		Therefore by (\cite{Kop-MR497932}, Lemma 1.10),  $\alpha_1 \in \Ep_{2n}(A)$. 
		\newline
		
		For a proof of the second part, observe that for any positive face $\tau$ of $\Gamma$ and $\m(\tau) = \m \cap R[M(\tau)] \in \Max(R[M(\tau)]),$ we have the inherited retraction:
		$$R[M(\tau)]_{\m(\tau)} \hookrightarrow B \xrightarrow{\pi_{\tau}} R[M(\tau)]_{\m(\tau)}, $$
		where $\pi_{\tau}\mid_{M(\tau)} = {\I}_{M(\tau)},$ and $\pi_{\tau}(x)=0$ if $x \notin M(\tau)$. Since $t$ is the generator of the $M$-submonoid $\mathcal{N}=M(\{P\}),$ owing to the relation shared by $P$ and $\Gamma$ we may infer
		$\pi_{\tau}(\gamma_m) = {\I}_{2n}$ for any positive face $\tau$ of $\Gamma$. Thus
		\begin{align*}
			\hspace{5cm} \gamma_m &\in \Ep_{2n}(A) \cap {\Sp}_{2n}(A, tB \cap A) \\
			& \subseteq \Ep_{2n}(A) \cap {\Sp}_{2n}(B, tB).
		\end{align*}
		Let  $$\gamma_m = \begin{pmatrix} 
			1 + ta_{11} & ta_{12} & \dots & ta_{1(2n)}\\
			ta_{12} & 1+ta_{22} & \ldots & ta_{2(2n)}\\
			\vdots & \ddots &  \ldots & \vdots\\
			ta_{(2n)1} & ta_{(2n)2} & \ldots  & 1+ta_{(2n)(2n)}
		\end{pmatrix},
		$$
		where $a_{ij} \in A$ for all $i,j$ and $ta_{ij} \in A$. Since $t$ is not invertible in $B,$ $ta_{ij} \in \m A,$ which makes the diagonals units in $A$. Let $\beta = \delta_I\gamma_m{\delta_I}^{-1}$. Then \begin{equation}\label{lab1}
			\hspace{5cm}\beta_{ij}= \left\{
			\begin{array}{ll}
				t(\gamma_m)_{ij} & i, \sigma(j) \in I, \\
				(1/t)(\gamma_m)_{ij} & \sigma(i), j \in I, \\
				(\gamma_m)_{ij} & {\rm else}.\\
			\end{array} 
			\right. 
		\end{equation} 
		Without loss of generality we may assume $2n \in I$. Since $A$ is local, there exists $\Delta_n \in \Ep_{2n}(A)$ such that $\Delta_n(\beta e_{2n}) = (1+a_{(2n)(2n)})e_{2n}$. Therefore we may assume the last column of $\beta$ is $(1+a_{(2n)(2n)})e_{2n}$. Next we prove $\beta \in {\Sp}_{2n}(B)$.
		To prove this we use the two relations: $t{\delta_I}^{-1}=\delta_{\sigma(I)}$ and $\psi_n\delta_{\sigma(I)} = \delta_{I}\psi_n$. Observe,
		\begin{align*}
			\hspace{5cm}\beta^T\psi_n\beta & = ({\delta_I\gamma_m{\delta_I}^{-1}})^T\psi_n(\delta_I\gamma_m{\delta_I}^{-1})\\
			& = {\delta_I}^{-1}\gamma_m^T{\delta_I}\psi_n(\delta_I\gamma_m{\delta_I}^{-1})\\
			& = {\delta_I}^{-1}\gamma_m^T (t{\delta_{\sigma(I)}}^{-1})(\delta_{\sigma(I)}\psi_n)\gamma_m{\delta_I}^{-1}\\
			& =t({\delta_I}^{-1}\psi_n{\delta_I}^{-1}) \\
			& =\psi_n.
		\end{align*}
		Let $a_i = 1+a_{(2i)(2i)} \in A$ for $1 \leq i \leq n$. Thus we may further say, 
		$$\beta = \begin{pmatrix} 
			\beta^{(n-1)} & x & \textbf{0}\\
			\textbf{0}^T    & a_n^{-1} & 0\\
			y^T & t^2a_{(2n)(2n-1)}  & a_n 
		\end{pmatrix},
		$$
		where $x, y \in B^n$ and $\beta^{(n-1)} \in {\Sp}_{2n-2}(B)$ is of the same form as $\beta$, \textit{i.e.}, $\beta_{ij}^{(n-1)}$'s have the same property corresponding to the set $I'=I \smallsetminus \{2n-1, 2n\} \in \mathbb{M}^{(n-1)},$ just like $\beta_{ij}$'s have in Equation (\ref{lab1}). 
		Then by  (\cite{Basuinjective-MR2578583}, Lemma 3.6), there exists $\Delta_{n-1} \in \Ep_{2n}(B)$ such that 
		$$\beta = \Delta_{n-1}\begin{pmatrix} 
			\beta^{(n-1)} & \textbf{0} & \textbf{0}\\
			\textbf{0}^T & a_n^{-1} & 0\\
			\textbf{0}^T & 0  & a_n
		\end{pmatrix}.$$
		Proceeding as before, by annihilating the off diagonal elements in decreasing order of the column index, we reach the index $i=2$. At $i=2,$ there exists $\Delta_1 \in \Ep_{2n}(B)$ and $\beta^{(1)} \in {\Sp}_2(B)$ such that
		$$\beta = \Delta_1 \begin{pmatrix} 
			\beta^{(1)} & \textbf{0} & \textbf{0} & \ldots & \textbf{0} & \textbf{0} \\
			\textbf{0}^T & a_2^{-1}   & 0 & \ldots & 0 & 0\\
			\textbf{0}^T & 0 & a_{2}   & \ldots    &  0 & 0\\
			\vdots & \vdots &  \vdots &  \ddots & \vdots & \vdots\\
			\textbf{0}^T & 0 & 0 & \ldots & a_n^{-1} & 0 \\
			\textbf{0}^T & 0 & 0  & \ldots & 0 & a_n
		\end{pmatrix}.$$
		As the diagonal elements of $\beta^{(1)}$ are units, therefore we may again reduce it to a diagonal matrix by right and left multiplication by elements of $\Ep_2(B)$. Since $a_i \in A$'s are units for all $i$, we get 
		$$\diag(a_1^{-1}, a_1, \ldots, a_n^{-1}, a_n) \in {\Sp}_{2n}(A) = \Ep_{2n}(A).$$
		Thus we may conclude $ \delta_I\gamma_m{\delta_I}^{-1} = \beta \in \Ep_{2n}(B)$. 
	\end{proof}
	
	Let $\Lambda$ be an ideal of $R$. By $sD(\Lambda)$ denote the diagonal matrices of ${\Sp}_{2n}(\Lambda)$. For $I \in \mathbb{M}^{(n)},$ we define $sI(\Lambda) = \{se_{ij}(\lambda) \in \Ep_{2n}(R) \mid i \in I, \lambda \in \Lambda \}$ and
	$s\widetilde{I}(\Lambda) = sI(\Lambda)\cdot sD(\Lambda) = sD(\Lambda)\cdot sI(\Lambda)$. With these notations, we fix $J \in \mathbb{M}^{(n)}$ and define
	$$sV:= \Ep_{2n}(B)\cdot s\widetilde{J}(C)\cdot \Ep_{2n}(C,\m C).\\$$ 
	
	For $n \geq 2,$ $1 \leq i \neq j\leq 2n$ and $u$ unit in $R$ define
	$$sw_{ij}(u)=se_{ij}(u)se_{\sigma(i)\sigma(j)}((-1)^{i+i}u^{-1})se_{ij}(u).$$
	Denote by $sW(R)$ the subgroup of $\Ep_{2n}(R)$ generated by these $sw_{ij}(\cdot)$'s. The reader may observe that elements of $sW(R)$ are monomial matrices. The following lemma gives us a way to tweak the results that we obtain for the ring $B$ to that of $B_-$. The proof of the following is straightforward.
	\begin{lemma}\label{l2}
		Let $I \in \mathbb{M}^{(n)},$ $c \in C$ and $1 \leq i \neq j \leq 2n$. Then
		$$ 
		\delta_Ise_{ij}(c){\delta_I}^{-1}= \left\{
		\begin{array}{ll}
			se_{ij}(tc) & i, \sigma(j) \in I, \\
			se_{ij}(t^{-1}c) & \sigma(i),j \in I,\\
			se_{ij}(c) & {\rm else}.\\
		\end{array} 
		\right. 
		$$
	\end{lemma}
	
	\begin{lemma}\label{l4}
		Let $n \geq 2$ and $I \in \mathbb{M}^{(n)}.$ Then $\delta_I{\ESp}_{2n}(B){\delta_I }^{-1} \subseteq sV$.
	\end{lemma}
	
	\begin{proof}
		Let $\alpha \in \Ep_{2n}(B)$. Consider the following composition of maps:
		$$A \hookrightarrow B \xrightarrow{\pi} B/tB,$$
		where the inclusion is natural and $\pi$ is the canonical surjection. Since $M$ is polarized, one gets $\pi(A)=\pi(B)$. Therefore we may choose a lift $\alpha_0 \in \Ep_{2n}(A)$ of $\bar{\alpha} \in B/tB$ such that $\alpha_{0} \equiv \alpha$ mod $tB$. Then $\alpha = \alpha'\alpha_0,$ for some $\alpha' \equiv {\I}_{2n}$ mod $tB$. By Lemma \ref{l1}, we may assume that $\alpha \in \Ep_{2n}(A)$. Denote by $\phi,$ the canonical homomorphism $\phi:A \rightarrow A/\m = k$. By Lemma 2.4 and Corollary 2.6 of \cite{Kop-MR497932}, we may decompose $\phi(\alpha)$ as $\phi(\alpha) = \beta_1\beta_2\beta_3,$
		where $\beta_1 \in sI(k),$ $\beta_2 \in sW(k)$ and $\beta_3 \in sJ(k)$. We may lift these $\beta_i$'s to $\alpha_i$'s by again applying (\cite{Kop-MR497932}, Lemma 2.4). Therefore we may write $\alpha$ as
		$$\alpha=\alpha_1\alpha_2\alpha_3\alpha_4,$$
		where $\alpha_1 \in sI(A),$ $\alpha_2 \in sW(A),$ $\alpha_3 \in sJ(A)$ and $\alpha_4 \in {\Sp}_{2n}(A, \m) = \Ep_{2n}(A, \m)$ (Since $A$ is local). By Lemma \ref{l2}, $\delta_I\alpha_4{\delta_I }^{-1} \subseteq \Ep_{2n}(C,\m C)$. As $\delta_I\alpha_2{\delta_I}^{-1} \in sW(C),$ there exists ${\bf\Pi} \in \Ep_{2n}(B)$ and $\Delta \in sD(C)$ such that,  $\delta_I\alpha_2{\delta_I}^{-1}={\bf{\Pi}}\Delta$. Then
		$$\hspace{4cm}\delta_I\alpha{\delta_I }^{-1} = \big(\delta_I\alpha_1{\delta_I }^{-1}{\bf\Pi}\big)\big(\Delta\delta_I\alpha_3{\delta_I }^{-1}\big)\big(\delta_I\alpha_4{\delta_I }^{-1}\big)\in sV. \hspace{2.2cm}\qedhere$$
	\end{proof}
	
	Note that by the normality of $\Ep_{2n}(C,\m)$ in ${\Sp}_{2n}(C, \m)$ (\cite{Kop-MR497932}, Corollary 1.11) and Lemma \ref{l2}, one may conclude that $\delta_IsV{\delta_I }^{-1} \subseteq sV.$
	
	\begin{corollary}\label{c2}
		Let $n \geq 2$. Then ${\ESp}_{2n}(C) = sV.$ 
	\end{corollary}
	
	\begin{proof}
		Clearly $sV \subseteq \Ep_{2n}(C)$. As $a$ varies over $A,$ the set $\{se_{ij}(at), se_{ij}(at^{-1}) \mid a \in A\}$ generates $\Ep_{2n}(C)$. Therefore it is sufficient to show that $se_{ij}(at), se_{ij}(at^{-1}) \subseteq sV$. For $a \in A,$ $se_{ij}(at) \in \Ep_{2n}(B)$. By Lemma \ref{l2}, for each $i, j,$ we may choose appropriate $I \in \mathbb{M}^{(n)}$ such that, $ se_{ij}(at^{-1}) = \delta_Ise_{ij}(a)\delta_I^{-1}$. The rest follows from Lemma \ref{l4}.
	\end{proof}
	
	\begin{proposition}\label{p3}
		Let $n \geq 2$. Then 
		$${\ESp}_{2n}(C) \cap {\Sp}_{2n}(C, \m C) \subseteq \big({{\ESp}_{2n}}(B) \cap {{\Sp}_{2n}}(B, \m B)\big)\cdot{{\ESp}_{2n}}(C,\m C).$$
	\end{proposition}
	
	\begin{proof}
		Let $\alpha \in \Ep_{2n}(C) \cap {\Sp}_{2n}(C, \m C)$. By Corollary \ref{c2}, there exists $\alpha_1 \in \Ep_{2n}(B),$ $\alpha_2 \in s\widetilde{J}(C)$ and $\alpha_3 \in \Ep_{2n}(C, \m C)$ such that 
		$$\alpha = \alpha_1\alpha_2\alpha_3.$$
		Since $M$ is polarized, therefore for an indeterminate $X,$ 
		$$C/\m C \simeq (A/\m A)[X,X^{-1}] = k[X,X^{-1}].$$ 
		Consider the canonical homomorphism $\phi: C \rightarrow C/\m C.$ Then 
		$${\I}_{2n} = \phi(\alpha) = \phi(\alpha_1)\phi(\alpha_2).$$
		Therefore $\phi(\alpha_1) = \phi(\alpha_2)^{-1} \in \Ep_{2n}(k[X]) \cap s\widetilde{J}(k[X,X^{-1}]) = s\widetilde{J}(k[X])$. By Lemma 2.4 of \cite{Kop-MR497932}, there exists $\alpha_1'\in s\widetilde{J}(B),$ such that $\phi(\alpha_1') = \phi(\alpha_1)$.	As ${\alpha_1'}\alpha_2 \in s\widetilde{J}(C)$ satisfying $\phi({\alpha_1'}\alpha_2)={\I}_{2n}$, on again applying Lemma 2.4 of \cite{Kop-MR497932}, we get ${\alpha_1'}\alpha_2 \in \Ep_{2n}(C, \m C)$. We may conclude the proof by observing:
		$$\hspace{3cm} \alpha=\big(\alpha_1{\alpha_1'}^{-1}\big)\big(\alpha_1'\alpha_2\alpha_3\big) \subseteq \big(\Ep_{2n}(B) \cap {\Sp}_{2n}(B, \m B)\big)\cdot\Ep_{2n}(C,\m C). \hspace{1cm} \qedhere $$
	\end{proof}
	
	\begin{definition}
		Given a ring $R,$ we define the following: 
		\begin{align*}
			\hspace{5cm}sG_1 &:= \Ep_{2n}(B) \cap {\Sp}_{2n}(B, \m B),\\
			sG_2 &:= \Ep_{2n}(B_-) \cap {\Sp}_{2n}(B_-, \m B_-),\\
			sG   &:= sG_1\cdot sG_2,\\
			sH   &:= \{h \in \Ep_{2n}(C) \mid h(sG)h^{-1} \subseteq sG\}.	
		\end{align*}
	\end{definition}
	
	\begin{lemma}\label{l3}
		Let $I \in \mathbb{M}^{(n)}$,  $a \in A$ and $1 \leq i \neq j \leq 2n$. Then 
		\begin{enumerate}
			\item $\delta_I(sG){\delta_I}^{-1}, {\delta_I}^{-1}(sG)\delta_I \subseteq sG.$
			\item  $se_{ij}(at),  se_{ij}(at^{-1}) \in sH.$
		\end{enumerate}
	\end{lemma}
	
	\begin{proof}
		Let $\alpha = \alpha_1\alpha_2 \in sG,$ where $\alpha_i \in sG_i$ for $i=1,2$. It is sufficient to show $\delta_I\alpha{\delta_I}^{-1} \in sG$. Let $\alpha_1 = \prod\limits_{k=1}^{l}se_{i_kj_k}(g_k),$ where $g_k \in B$ for all $k$. Since $M$ is a polarized monoid, we may write $g_k = a_k+tf_k,$ where $a_k \in A,$ $f_k \in B$. For $1 \leq p \leq l,$ fix $\beta_p = \prod\limits_{k=1}^{p}se_{i_kj_k}(a_k)$. Then 
		$$\alpha_1 = \Bigg(\prod\limits_{k=1}^{l}\beta_kse_{i_kj_k}(tf_k)\beta_k^{-1} \Bigg)\beta_l \in \big(\Ep_{2n}(B) \cap {\Sp}_{2n}(B, t\m B)\big)\cdot\Ep_{2n}(A, \m A)$$
		
		\noindent By (\cite{Gubeladze-ClassicalMR1079964}, Proposition 1.14) , $(P_-, \Gamma, M_-)$ is also a polarized monoid. Following a similar procedure as before, we may say $\alpha_2 \in \Ep_{2n}(A, \m)\cdot\big(\Ep_{2n}(B_-) \cap {\Sp}_{2n}(B_-, t\m B_-)\big)$. Combining these conclusions, we write
		$$\alpha=\alpha_1'\alpha_0\alpha_2',$$
		where $\alpha_1' \in \Ep_{2n}(B) \cap {\Sp}_{2n}(B, t\m B), \alpha_2' \in \Ep_{2n}(B_-) \cap {\Sp}_{2n}(B_-, t^{-1}\m B_-)$ and $\alpha_{0} \in \Ep_{2n}(A, \m)$. By emulating the proof of Lemma \ref{l1}, we have $\delta_I\alpha_i{\delta_I}^{-1} \in sG_i$ for $i=1,2$.
		
		Next we need to show that $\delta_I\alpha_0{\delta_I}^{-1} \in sG$. Since $\alpha_{0} \in \Ep_{2n}(A, \m A),$ where $A$ is a local ring, the diagonal elements of $\alpha_0$ are congruent to $1$ modulo $\m$ and therefore are invertible in $A$. Without loss of generality we may assume $2n \in I$. We may find $\beta \in sI(\m)$ such that $e_{2n}^T\beta\alpha_0=(0, 0, \ldots, {\alpha_0}_{(2n)(2n)})$. Let $\sigma(I):=I^c \cap \{1, \ldots, 2n\} \in \mathbb{M}^{(n)}$. We may choose $\beta' \in s\sigma(I)(\m)$ such that $e_{2n-1}^{T}\beta\alpha_0\beta' = (0, 0, \ldots, {{\alpha_0}_{(2n)(2n)}}^{-1} , 0)$. Continuing this way, we may claim that 
		$$\alpha_0 = \beta_0\big(\diag({{\alpha_0}_{11}}^{-1},{\alpha_0}_{11}, \ldots, {{\alpha_0}_{(2n)(2n)}}^{-1}, {\alpha_0}_{(2n)(2n)})\big)\beta_0',$$
		where $\beta_0 \in sI(\m), \beta_0' \in s\sigma(I)(\m)$ and $\diag({{\alpha_0}_{11}}^{-1},{\alpha_0}_{11}, \ldots, {{\alpha_0}_{(2n)(2n)}}^{-1}, {\alpha_0}_{(2n)(2n)}) \in \Ep_{2n}(A)$. Using Lemma \ref{l2} we may say:
		\begin{align*}
			\delta_I\beta_0~\diag({{\alpha_0}_{11}}^{-1},{\alpha_0}_{11}, \ldots, {{\alpha_0}_{(2n)(2n)}}^{-1}, {\alpha_0}_{(2n)(2n)}){\delta_I}^{-1} \in &~\Ep_{2n}(B, \m B)  \subseteq sG_1\\
			\delta_I\beta_0'{\delta_I}^{-1} \in &~\Ep_{2n}(B_-, \m B_-) \subseteq sG_2.
		\end{align*}
		Thus $\delta_I\alpha_0{\delta_I}^{-1} = \bigg(\delta_I\big(\beta_0~\diag(a_1, a_1^{-1}, \ldots, a_n, a_n^{-1})\big){\delta_I}^{-1}\bigg)\bigg(\delta_I\beta_0'{\delta_I}^{-1}\bigg) \in sG_1\cdot sG_2 = sG$.
		
		To prove the second part, it is sufficient to prove $se_{ij}(at^{-1}) \in sH$ as by invoking the property of polarized monoid $M_-$, we may prove its counterpart. Choose an $I \in \mathbb{M}^{(n)}$ such that $\sigma(i), j \in I$. Since $\delta_I \in sH,$ therefore there exists $\beta_i \in sG_i$ such that $\delta_I\alpha_1\alpha_2{\delta_I}^{-1} = \beta_1\beta_2$. By Lemma \ref{l2}, we have
		\begin{align*}
			se_{ij}(at^{-1})\alpha se_{ij}(at^{-1})^{-1} &= se_{ij}(at^{-1})\alpha_1\alpha_2se_{ij}(-at^{-1})\\
			&=\big({\delta_{I}}^{-1}se_{ij}(a)\delta_{I}\big)\alpha_1\alpha_2\big({\delta_{I}}^{-1}se_{ij}(-a)\delta_{I}\big)\\
			&=\bigg({\delta_{I}}^{-1}\big(se_{ij}(a)\beta_1se_{ij}(-a)\big){\delta_{I}}\bigg)\bigg({\delta_{I}}^{-1}\big(se_{ij}(a)\beta_2se_{ij}(-a)\big){\delta_{I}}\bigg)\\
			&=\bigg(t^{-1}{\delta_{\sigma(I)}}\big(se_{ij}(a)\beta_1se_{ij}(-a)\big)t{{\delta_{\sigma(I)}}^{-1}}\bigg)\bigg(t^{-1}{\delta_{\sigma(I)}}\big(se_{ij}(a)\beta_2se_{ij}(-a)\big)t{{\delta_{\sigma(I)}}^{-1}}\bigg)\\
			&\in sG_1\cdot sG_2 = G. \qedhere
		\end{align*}
	\end{proof}
	
	\begin{corollary}\label{c1}
		$sH={\ESp}_{2n}(C).$
	\end{corollary}
	\begin{proof}
		This follows directly by Lemma \ref{l2} and \ref{l3} by observing that  $\{se_{ij}(at), se_{ij}(at^{-1}) \mid a \in A\}$ forms a generating set for $\Ep_{2n}(C)$. 
	\end{proof}
	
	\begin{lemma}\label{l5}
		${\ESp}_{2n}(C) \cap {\Sp}_{2n}(C, \m C) \subseteq sG$.
	\end{lemma}
	
	\begin{proof}
		Let $\beta \in \Ep_{2n}(A, \m),$ $\gamma \in \Ep_{2n}(B_-,\m B_-)$ and $\delta \in \Ep_{2n}(C)$. Observe that elements of the form $\delta^{-1}\beta\delta$ and $\delta^{-1}\gamma\delta$ generate $\Ep_{2n}(C, \m C)$.  Using this fact and Corollary \ref{c1}, we show that ${\ESp}_{2n}(C, \m C)sG \subseteq sG$:
		\begin{align*}
			\hspace{3.5cm}\delta^{-1}\beta\delta sG &\subseteq \delta^{-1}\beta sG\delta \subseteq \delta^{-1} sG\delta \subseteq sG,\\
			\delta^{-1}\gamma\delta sG &\subseteq \delta^{-1}\gamma\delta \big( (\gamma\delta)^{-1} \big (\gamma\delta)) \subseteq \delta^{-1} sG \gamma \delta \subseteq \delta^{-1} sG \delta \subseteq sG.
		\end{align*}
		
		\noindent Let $\Delta \in {\ESp}_{2n}(C) \cap {\Sp}_{2n}(C, \m C)$. By Proposition $\ref{p3},$ we have $\Delta=\alpha_1\alpha_2,$ where $\alpha_1 \in sG_1 \subseteq sG$ and $\alpha_2 \in {\ESp}_{2n}(C, \m C)\subseteq sG$.
	\end{proof}
	
	\begin{proposition}\label{p1}
		Let $n \geq 2$ and $\alpha \in {\Sp}_{2n}(B)$ be such that there exists $\beta \in {\Sp}_{2n}(B_-),$ satisfying $\alpha\beta^{-1} \in {\ESp}_{2n}(C)$. Then $\alpha \in {\ESp}_{2n}(B)$. 
	\end{proposition}
	
	\begin{proof}
		As both $M$ and $M_-$ are polarized monoids, we have:
		$$B/\m B \simeq (A/\m)[X] \simeq k[X] \text{ and } B_-/\m B_- \simeq (A/\m)[X^{-1}] \simeq k[X^{-1}].$$
		Since $\det{(\alpha)} = \det(\beta$), following a similar procedure as in Lemma \ref{l1}, we may say that $\alpha \equiv 1$ mod $\m B$ and $\beta \equiv 1$ mod $\m B_-$. Therefore $\alpha\beta^{-1} \in \Ep_{2n}(C) \cap \Sp_{2n}(C,\m C)$. From Lemma \ref{l5}, we have $\alpha\beta^{-1} \in sG$. Choose $\alpha_i \in sG_i$ for $i=1,2$ such that $\alpha\beta^{-1} = \alpha_1\alpha_2$. Then
		$$\Delta = \alpha_+^{-1}\alpha = \alpha_2\beta \in {\Sp}_{2n}(B) \cap {\Sp}_{2n}(B_-).$$
		Since $A$ is a local ring and $B \cap B_- = A$,  we may infer $\Delta \in \Ep_{2n}(A)$. The result follows by observing
		$$\hspace{4cm}\alpha=\alpha_+(\alpha_+^{-1}\alpha) \in \Ep_{2n}(B)\Ep_{2n}(A) \subseteq \Ep_{2n}(B). \hspace{2.8cm} \qedhere$$
	\end{proof}

	\begin{remark}\label{r4}
		Denote by $\mathbf{H}_r$ the (induction) hypothesis for the Theorem \ref{t1} for $c$-divisible monoid of $\rank r$. 
	\end{remark}
	
	For a monoid $L,$ if $\phi(L)$ is a closed polytope, then $\phi(L)$ can be decomposed as 
	$$\phi(L) = \delta \cup \gamma, $$
	where $\delta$ is a pyramid, $\gamma$ a closed polytope and $\delta \cap \gamma$ is a face of both $\delta$ and $\gamma$. Also 
	$$\dim(\gamma) = \dim(\delta) = \dim(\phi(L)) \text{ and } \dim(\gamma \cap \delta) = \dim(\phi(L)) - 1.$$  
	Fix $\delta' \subsetneq \delta$ a  subpyramid such that the apex of $\delta'$ is in $\intt(\delta)$ and the base of both these pyramids coincide.

	The proof of the following runs similar to that of Lemma 2.18 in \cite{Gubeladze-ClassicalMR1079964}. It is primarily a geometric proof, the algebraic components of which are satisfied by the presence of Proposition \ref{p1}, (\cite{Kop-MR497932}, Lemma 4.2) and (\cite{Basuinjective-MR2578583}, Proposition 3.10). We give a condensed form of this proof indicating the crucial (algebraic) steps.

	\begin{lemma}\label{l6}
		Let $R \in \mathcal{R}_k$ be a local ring, $L$ a normal $c$-divisible monoid of rank $r$ for which $\phi(L)$ is a closed polytope, $k \geq 2$ and $c>1$. If $\mathbf{H}_{r-1}$ holds, then for $\alpha \in {\Sp}_{2k}(R[L_*]),$ there exists $\beta \in {\ESp}_{2k}(R[L_*])$ such that
		$\beta\alpha \in {\Sp}_{2k}(R[L(\delta'\cup\gamma)_*])$.
	\end{lemma}
	
	\begin{proof}
		By the use of $\mathbf{H}_{r-1}$ and (\cite{Kop-MR497932}, Lemma 4.2) we are able to choose a polarized monoid $(P_0, \Gamma_0, M)$ such that
		\begin{enumerate}[(i)]
			\item $\Gamma_0 \subseteq \intt(\Gamma)$, where $\Gamma=\delta'\cup \gamma$
			\item $M \subseteq L_*$,
			\item $\alpha \in {\Sp}_{2k}(R[M])$ and
			\item $\alpha_t \in \Ep_{2k}(R[\mathcal{N}_0^{-1}M]),$ where $\mathcal{N}_0=M({P_0}) \simeq \pz{}t$.
		\end{enumerate} 
		Next by the use of Proposition \ref{p1} and (\cite{Basuinjective-MR2578583}, Proposition 3.10), as in the proof of Lemma 2.18 in \cite{Gubeladze-ClassicalMR1079964}, we find a non-zerodivisor $s \in R[M]$ such that $\alpha_s \in \Ep_{2k}(R[M]_s),$ and for sufficiently large $l$ and $k$ define the $R[M]$-homomorphism,
		$$\theta: R[M]_s[Z] \rightarrow R[M]_s, $$
		as $\theta(Z) = (-1+(a_1m_1 + \cdots + a_qm_q)^{2^l})/s^{k},$ where $s^k=1+a_1m_1 + \cdots + a_qm_q,$ $a_i \in R$  and $m_i \in M(\Gamma_0)$ for all $i$. The element $\beta$ is chosen as
		$$\beta=(\alpha\alpha')^{-1},$$
		where $\alpha' = \theta(\alpha_s(1+s^kZ)t) \in \Ep_{2k}(R[M]) \subseteq \Ep_{2k}(R[L_*]).$
	\end{proof}

	\begin{proposition}\label{p2}
		Let $R$  be a ring, $L$ a normal positive c-divisible monoid of rank $r',$ with $\phi(L)$ a closed polytope and $n >0$. If
		\begin{enumerate}[(i)]
			\item ${\Sp}_{2n}(R) = {\ESp}_{2n}(R)$
			\item ${\Sp}_{2n}(R_{\m}[L]) = {\ESp}_{2n}(R_{\m}[L])$ for all $\m \in \Max(R)$.
		\end{enumerate}
		Then ${\Sp}_{2n}(R[L]) = {\ESp}_{2n}(R[L])$.
	\end{proposition}
	
	\begin{proof}
		Let $\m \in \Max(R)$ and $\alpha \in  {\Sp}_{2n}(R[L])$. By $(ii)$ there exists $r \in R \smallsetminus \m,$ such that $\alpha_{r} \in \Ep_{2n}(R_r[L])$. Since $R$ has a finite cover, we may choose $r_i \in R$ such that $(r_1, \ldots, r_k) =1$ and $\alpha_{r_i} \in \Ep_{2n}(R_{r_i}[L])$ for $1 \leq i \leq k$. Choose an affine $L$-submonoid $L'$ such that $\alpha \in {\Sp}_{2n}(R[L'])$ satisfying $\alpha_{r_i} \in \Ep_{2n}(R_{r_i}[L'])$ for all $i$. By Proposition 1.7 and 1.15 of \cite{Gubeladze-ClassicalMR1079964} there exists a simplex $\Delta$ enveloping $\phi(L)$ such that
		$$L' \subset L \subset \mathbb{Q}_+ \otimes \Delta(L) \simeq \mathbb{Q}_+^{r'}.$$
		Since $L'$ is affine, we may embed $R[L'] \hookrightarrow R[X_1, \ldots, X_{r'}]$ (\cite{Gubeladze1-bookMR2508056}, Proposition 4.16) and therefore $R[L']$ has an inherited positive $\pz{r'}$-grading with the zeroth homogenous component as $R$. Since $(i)$ holds, we may assume $\alpha \in {\Sp}_{2n}(R[L'], R[L']^+)$. On invoking Theorem 8.2 of \cite{Rab-Kun}, $\alpha \in \Ep_{2n}(R[L']) \subseteq \Ep_{2n}(R[L])$.
	\end{proof}

	{Now we are in a position to prove the main proof of this section.}
	
	\begin{proof}[\textbf{Proof of Theorem \ref{t1}}]
		
		Let $\alpha \in {\Sp}_{2k}(R[L])$. In the first part of the proof we show that it is sufficient to assume that $L$ is a finite rank positive monoid with $\phi(L)$ closed polytope.
		
		Let $L' = \{L \smallsetminus U(L)\} \cup \{0\}$. Then $L'$ is a positive $c$-divisible monoid. Consider the retraction map $\psi: R[L] \rightarrow R[U(L)],$  given by $\alpha\mid_{U(L)}={\I}_{U(L)}$ and $\alpha(l)=0$ for all $l \notin U(L)$. Then
		
		$$\psi(\alpha) \in {\Sp}_{2k}(R[U(L)]) \simeq {\Sp}_{2k}(R[\nz{k}]) = \Ep_{2k}(R[\nz{k}]) \simeq \Ep_{2k}(R[U(L)]).$$
		
		\noindent Then for $\alpha'=\psi(\alpha)^{-1} \in {\Sp}_{2k}(R[U(L)])$ we get $\alpha'\alpha \in {\Sp}_{2k}(R[L'])$. Therefore, we may assume $L$ is positive. Since $L$ can be seen as a direct limit of $c$-divisible monoids of finite rank $L_i$ with $\phi(L_i)$ closed polytope ($cf.$ Remark \ref{r1}), we may also assume $L$ to be of finite rank $r$ and $\phi(L)$ closed polytope.

		The proof comes through by induction on $r$ (see $\mathbf{H}_r$ Remark \ref{r4}). The $r=1$ case holds, as in this case $L$ turns out to be the direct limit of $\pz{}$ or $\nz{},$ as per the availability of non-trivial units. Let $r>1$. Next we descend into the interior of $M$ so that we may enforce normality on $L$.
		
		Since $\phi(L)$ is a closed polytope, we may fix a vertex $v_1$ of $\phi(L)$ and the submonoid $L'=L(\{v_1\}) \simeq \pz{}$. Let $\pi_1$ be the retraction $\pi_1$ from $R[L]$ to $R[L']$ given by $\pi_1\mid_{L'}={\I}_{L'}$ and $\pi_1(l)=0$ for all $l \notin L'$. Then as before, we choose $\beta_1=\pi_1(\alpha)^{-1} \in \Ep_{2k}(R[L'])$ and $L_1 = L(\phi(L) \smallsetminus \{v_1\})$. Then $\beta_1\alpha \in {\Sp}_{2k}(R[L_1])$. Fix a vertex $v_2$ of $\phi(L_1)$. Proceeding as before by defining $L_1' = L_1(\{v_2\})$ we  end up with  $\beta_2 \in \Ep_{2k}(R[L_1'])$ and $L_2 = L_1(\phi(L_1) \smallsetminus \{v_2\}) = L(\phi(L) \smallsetminus \{v_1, v_2\})$, such that $\beta_2(\beta_1\alpha) \in {\Sp}_{2k}(R[L_2])$. Since we have finite vertices, there exists a $\beta' \in \Ep_{2k}(R[L])$ such that $\beta'\alpha \in {\Sp}_{2k}(R[L_v]),$ where $L_v := L(\phi(L) \smallsetminus \{\text{all vertices of } \phi(L)\} )$. 
		
		\noindent  Let $F$ be any $1$-dimensional face of $\phi(L)$. Since $L(F_*)$ is the direct limit of normal $c$-divisible monoids and $\pi: R[L_v] \rightarrow R[F_*]$ is a retraction, $\pi(\beta'\alpha) \in {\Sp}_{2k}(R[F_*]) = \Ep_{2k}(R[F_*])$ (Since $\mathbf{H}_{r-1}$ implies $\mathbf{H}_2$). In this way, by the use of $\mathbf{H}_{k}$ for $k<r$ we may cleave off the boundary from $L$ to take our calculations to the level of the interior monoid $L_*$. We may thus say that there exists $\alpha' \in \Ep_{2k}(R[L])$ such that $\alpha'\alpha \in {\Sp}_{2k}(R[L_*])$.  Since $L_*$ is a direct limit of normal $c$-divisible monoids, we may assume that $L$ is normal. If $\phi(L)$ is open, then we may write $L=\varprojlim L_i,$ where $L_i$'s are positive $c$-divisible $L$-submonoids for which $\phi(L_i)$'s are closed. Since ${\ESp}_{2k}(R) = {\Sp}_{2k}(R)$, by Proposition \ref{p2} we may assume that $R$ is local. 
		
		By Lemma 2.8 of \cite{Gubeladze2-maximalMR937805} and Lemma \ref{l6}, we may choose a rational simplex $\Delta \subset \phi(L)$ such that $\dim(\Delta)=\dim(\phi(L))$ and $\alpha'' \in \Ep_{2n}(R[L_*])$ such that $\beta:=\alpha''\alpha'\alpha \in {\Sp}_{2n}(R[L(\Delta)_*])$. By Approximation theorem A of \cite{Gubeladze-ClassicalMR1079964}, we have $L(\Delta)_*$ is the direct limit of the free monoid $\pz{r}$.
	\end{proof}
	
	\begin{theorem}\label{t3}
		Let $R$ be a regular ring of dimension $d,$ ${\KK}_1{\Sp}(R)=0$ and $L$ be a c-divisible monoid, where $c > 1$. For $k \geq d+2,$
		$${\ESp}_{2k}(R[L]) = {\Sp}_{2k}(R[L]).$$ 
	\end{theorem}
	
	\begin{proof}
		Since ${\KK}_1{\Sp}(R)=0,$ by the Stabilization theorem in \cite{Vaserstein-MR0269722}, for $k \geq d+2,$ we get 
		$${\ESp}_{2k}(R) = {\Sp}_{2k}(R).$$
		Building upon the availability of Corollary \ref{cktheorysp}, $L$ can be assumed to be a finite rank positive normal monoid with $\phi(L)$ a closed polytope (as done in the proof of Theorem \ref{t1}). Since $R_{\m} \in \mathcal{R}_k,$ for all $\m \in \Max(R),$ therefore by Proposition \ref{p2} and Theorem \ref{t1} we have the required.
	\end{proof}

	From Example \ref{r3} and \ref{r6}, we similarly prove the following:
	
	\begin{theorem}\label{t7}
		Let  $R$ be a regular ring of dimension $d$ and $L$ be a c-divisible monoid, where $c > 1$. If 
		\begin{enumerate}[(i)]
			\item $d \leq 1,$ then ${\Sp}_{2k}(R[L]) = {\Sp}_{4}(R){\ESp}_{2k}(R[L])$ for all $k \geq 2$.
			\item R is a geometrically regular ring containing a field with trivial ${\KK}_1{\Sp}$, then ${\Sp}_{2k}(R[L]) = {\ESp}_{2k}(R[L])$ for all $2k \geq \max\{3,d+1\}$.
		\end{enumerate}
	\end{theorem}
	
	We are now in a position to prove ${\KK}_1{\Sp}$-invariance:
	
	\begin{proof}[\textbf{Proof of Theorem \ref{t2}}]
		Let $A=R[L]$ and consider the stable symplectic and elementary symplectic group
		$${\Sp}(A)=\bigcup\limits_{l \geq 1} {\Sp}_{2t}(A) \text{ and } {\ESp}(A) = \bigcup\limits_{l \geq 1} {\ESp}_{2t}(A).$$
		Assume $\alpha \in {\Sp}(A)$. Without loss of generality we may assume $\alpha \in {\Sp}_{2k}(A)$ for some $k \geq \dim(R)+2$. Since $R$ is regular, on application of Theorem \ref{t3},
		$$\alpha_{\m} \in  {\Sp}_{2k}(A_{\m}) = {\ESp}_{2k+2}(A_{\m}).$$
		\noindent Proceeding as in the proof of Theorem \ref{t1},  we may assume $L$ to be positive of finite rank $r$ and $\phi(L)$ closed polytope.  As before, we descend into the interior of $L$ to obtain $\alpha \in {\Sp}_{2k}(R[L_*]){\Ep}_{2k}(R[L])$. Therefore we may assume that $L$ is normal. If $\phi(L)$ is a open, then we may write $L=\varprojlim L_i,$ where $L_i$'s are positive $c$-divisible $L$-submonoids for which $\phi(L_i)$'s are closed. This allows us to choose an affine $L$-submonoid $L'$ such that $\alpha \in {\Sp}_{2k}(R[L'])$ satisfying $\alpha_{\m} \in \Ep_{2k}(R_{\m}[L'])$ for all $\m \in \Max(R)$. Here $R[L']$ has an inherited positive grading with $R[L']_0=R$. On invoking Theorem 8.2 of \cite{Rab-Kun},
		$$\alpha \in {\Sp}_{2k}(R){\Ep}_{2k}(R[L']) \subset {\Sp}_{2k}(R){\Ep}_{2k}(R[L]).$$
		Therefore ${K}_1Sp(R[L]) \simeq {K}_1Sp(R).$
	\end{proof}
	
	Let $M$ be any monoid, $A=R[M]$ and $A' = R[X_1, \ldots, X_m, Y_1^{\pm{1}}, \ldots, Y_n^{\pm{1}}]$. Given that monoid algebras are natural generalization of Laurent polynomial rings $A'$, the best general estimates that exist in literature for surjective and injective bounds of $R[M]$ usually matches upto that of its well studied predecessor $A'$.  
	For rings of the type $A',$ surjective bounds are due to (\cite{Suslin-MR0472792}, \cite{Lindel-MR1322406}, \cite{Kop-MR497932}, \cite{Keshari-Symp-MR2295082}, \cite{Bhat-symp-MR1858341}). As anticipated, best possible surjective bounds for monoid algebras were established in (\cite{Gubeladze-unimodularMR3853049}, \cite{Maria-MKK1}) for the linear case and in \cite{Rab-mar1} for the symplectic case. Upto this point, the theory for injective stabilization has only progressed till Laurent polynomial rings, where for $k \geq \max\{3, d+2\}$ (linear \cite{Suslin-MR0472792}) and $2k \geq \max\{4, 2d+4\}$ (symplectic \cite{Kop-MR497932}) the map $\phi_k$ becomes injective. For monoid algebras generated by $c$-divisible monoids, we have the following theorem:
	
	\begin{corollary}
		Let $R$ be a regular ring of dimension $d$ with ${\KK}_1{\Sp}(R)=0$ and $L$ be a $c$-divisible monoid. Then the map $\phi_k$ is an isomorphism for all $k \geq d+2$:
		$$\phi_k: \frac{{\Sp}_{2k}(R[L])}{{\ESp}_{2k}(R[L])} \xrightarrow{} {K}_1Sp(R[L]).$$
		Further, if $R$ is a special PID, then $\phi_k$ is an isomorphism for $k \geq 2$.
	\end{corollary}
	
	\section{Equality of symplectic transvections and isometeries}
	
	Let $R$ be a ring. The pair $(P, \lrinn{}{})$ is called a symplectic $R$-module, if $P$ is a projective $R$-module and $\lrinn{}{}$ a non-degenerate alternating bilinear form on $P$. The form `$\psi_n$' defined in Section \ref{s2} gives the standard (non-degenerate alternating) form on the free $R$-module $R^{2n}$. We call $(P,\lrinn{}{})$ the free symplectic $R$-module, if $P$ is free with the standard form. We only consider finitely generated modules in this section.
	
	Given two symplectic $R$-modules $(P, \lrinn{}{}_1)$ and $(Q, \lrinn{}{}_2)$ one may define a non-degenerate alternating bilinear form on $P \oplus Q$ given by
	$$\lrinn{p_1+q_1}{p_2+q_2} := \lrinn{p_1}{p_2}_1 + \lrinn{q_1}{q_2}_2.$$
	This combination of symplectic modules is denoted as $(P \perp Q, \lrinn{}{})$. We prefer to omit the distinguishing subscripts on the inner products, whenever the distinction between the inner products is clear from the context.
	
	Consider a symplectic $R$-module $(P, \lrinn{}{})$ and $\m \in \Max(R)$. One may define a corresponding local symplectic $R_{\m}$-module by  $(P_{\m}, \lrinn{}{}_{\m}),$ where
	$$ \Lrinn{\frac{p}{s}}{\frac{q}{s'}}_{\m}:= \frac{1}{ss'}\lrinn{p}{q},$$
	for $p, q \in P$ and $s, s' \in R \smallsetminus \m$.
	
	An isomorphism $\theta: (P, \lrinn{}{}) \rightarrow (Q, \lrinn{}{})$  of symplectic $R$-modules is a form preserving $R$-isomorphism between $P$ and $Q,$ \textit{i.e.,} $\theta: P \rightarrow Q$ is an isomorphism of $R$-modules satisfying:
	$$\langle p_1, p_2 \rangle = \langle \theta(p_1), \theta(p_2) \rangle,$$
	for $p_1, p_2 \in P$. These form preserving automorphisms of $(P,\lrinn{}{})$ are denoted by ${\Sp}(P)$. Following the definitions of Bass, we define the elementary transvections on $P \perp R^2,$ denoted by $\Etrans_{\Sp}(P \perp R^2)$, as the subgroup of $\Sp(P \perp R^2)$ generated by the automorphisms $\Delta_q$ and $\Gamma_q$:
	\begin{eqnarray*}
		\Delta_q:& ( p,a,b) ~\mapsto~(p+bq, a-\lrinn{p}{q}+b, b),\\
		\Gamma_q:& (p,a,b)~\mapsto~(p + aq, a, b+\lrinn{p}{q}-a),
	\end{eqnarray*}
	where $p, q \in P$ and $a,b \in R$. On fixing $q \in P,$ define $\phi_q \in {\Hom}_R(P,R)$ as $\phi_q(p) = \lrinn{p}{q},$ for $p \in P$.
	
	\begin{remark}\label{r2}
		Let $M$ be a monoid and $(P, \phi)$ a symplectic $R[M]$-module such that $P \simeq R[M]^{2k}$ for some $k>0$. Then by  (\cite{Lam-MR2235330}, Theorem 5.8) we have
		$$(P, \phi) \simeq (P, \psi_k),$$
		as symplectic $R[M]$-modules.
	\end{remark}
	
	\begin{proposition}\label{p4}
		Let $R$ be a regular ring of dimension d with ${\KK}_1{\Sp}(R)=0$ and $A=R[X_1, \ldots, X_m, Y_1^{\pm{1}}, \ldots, Y_{n}^{\pm{1}}]$. Let $P$ be a symplectic $A$-module of rank $2k,$ where $k \geq \max\{2, d+1\}$. Then
		$${\Etrans}_{\Sp}(P \perp A^{2}) = {\Sp}(P \perp A^2).$$
	\end{proposition}
	
	\begin{proof}
		From Corollary \ref{cktheorysp} we have the following:
		\begin{enumerate}
			\item ${\Sp}_{2l}(A) = {\ESp}_{2l}(A)$ for $l\geq d+2$,
			\item 	$P$ is stably free, $i.e.,$ $P \perp A^{2t} \simeq A^{2k+2t},$ for some $t\geq0$.
		\end{enumerate}

		\noindent If $t=0,$ we are done by (\cite{BBR-MR2578513}, Lemma 2.20). Let $\alpha \in {\Sp}(P \perp A^2)$. Then
		\begin{align*}
			\hspace{3cm}\alpha \perp {\I}_{2t-2} &\in {\Sp}(P \perp A^{2t})\\
			&= {\Sp}_{2(k+t)}(A)\\
			&= \Ep_{2(k+t)}(A) \hspace{2cm}\\
			&= {\Etrans}_{\Sp}(P \perp A^{2t})  \hspace{1cm} (\text{By (\cite{BBR-MR2578513}, Lemma 2.20})
		\end{align*}
		
		\noindent From the Stabilization Theorem in \cite{Kop-MR497932}, the natural map  $\Phi_{\Sp(l^+)}: \frac{{\Sp}(P \perp R^2)}{{\Etrans}_{\Sp}(P \perp R^2)} \rightarrow  \frac{{\Sp}(P \oplus R^{2l})}{{\Etrans}_{\Sp}(P \perp R^{2l})}$ is an isomorphism for all $l \geq 1$. Thus  $\alpha \in {\Etrans}_{\Sp}(P \perp R^{2}).$ 
	\end{proof}
	
	On application of local-global theorem for transvection group, the results of Section \ref{s3} can be extended to prove Theorem \ref{t6}:	
	
	\begin{proof}[\textbf{Proof of Theorem \ref{t6}}]
		Let $A=R[L]$ and $\alpha \in {\Sp}(P \perp A^2)$. By Proposition \ref{p4} and descending into the interior (as in Theorem \ref{t1}), we can assume $L$ is a finite rank positive normal $c$-divisible monoid with $\phi(L)$ closed polytope. 
		
		Let $\m \in \max(R)$. Since $R$ is regular, by Corollary 1.4 of \cite{Swan-andersonMR1144038} and Remark \ref{r2}, $P_{\m}$ is a free symplectic $A_{\m}$-module. On application of Theorem \ref{t3},
		$$\alpha_{\m} \in  {\Sp}(A_{\m}^{2k+2}) = {\Sp}_{2k+2}(A_{\m}) =  {\ESp}_{2k+2}(A_{m}).$$
		\noindent Therefore, there exists $r \in R \smallsetminus \m,$ such that $\alpha_{r} \in \Ep_{2k+2}(A_r)$. Since $R$ has a finite cover, we may choose $r_i \in R$ such that $(r_1, \ldots, r_k) =1$ and $\alpha_{r_i} \in \Ep_{2k+2}(A_{r_i})$ for $1 \leq i \leq k$. Choose an affine $L$-submonoid $L'$ such that $\alpha \in {\Sp}_{2k+2}(R[L'])$ satisfying $\alpha_{r_i} \in \Ep_{2k+2}(R_{r_i}[L'])$ for all $i$ and $P$ is a symplectic $R[L']$-module. By Proposition 1.7 and 1.15 of \cite{Gubeladze-ClassicalMR1079964} there exist a simplex $\Delta$ enveloping $\phi(L)$ such that
		$$L' \subset L \subset \mathbb{Q}_+ \otimes \Delta(L) \simeq \mathbb{Q}_+^{r'}.$$
		
		\noindent Here $R[L']$ has an inherited positive $\pz{r'}$-grading with $R[L']_0=R$. From (\cite{BBR-MR2578513}, Lemma 2.20), 
		$$\alpha_{\m} = (\alpha_{r})_{\m}\in  {\ESp}_{2k+2}(R_{\m}[L']) = {\Etrans_{\Sp}(P_{\m})}.$$
		Since ${\KK}_1{\Sp}(R)=0,$ we may assume  $\alpha \in {\Sp}(R[L'], R[L']_+)$. By (\cite{Rab-man}, Theorem 5.4) we have the required.
	\end{proof}
	
	Using Theorem \ref{t7}, we show the corresponding result for transvections:
	
	\begin{theorem}\label{t8}
		Let  $R$ be a regular ring of dimension $d,$ ${\KK}_1{\Sp}(R)=0$ and $L$ be a c-divisible monoid, where $c > 1$.  Let $P$ be a symplectic module of rank $2k$. Then ${\Etrans}(P \perp R[L]^2) = {\Sp}(P \perp R[L]^2)$, if
		\begin{enumerate}[(i)]
			\item $d \leq 1,$ ${\ESp}_{4}(R) = {\Sp}_{4}(R)$  and $k \geq 1$.
			\item R is a geometrically regular ring containing a field and $2k \geq \max\{1, d-1\}$. 
		\end{enumerate}
	\end{theorem}

	\section{Auxiliary results}
	
	In the hypothesis of Theorem \ref{t1}, we had to enforce the condition $${\ESp}_{2k}(R[X_1, \ldots, X_m, Y_1^{\pm{1}}, \ldots, Y_n^{\pm{1}}]) = {\Sp}_{2k}(R[X_1, \ldots, X_m, Y_1^{\pm{1}}, \ldots, Y_n^{\pm{1}}])$$ for $R \in \mathcal{R}_k$. In order to remedy this, we define a much general class of ring $\mathcal{R}'_n,$ which automatically satisfies this requirement. It can be shown that the conclusion of Theorem \ref{t1} also holds for rings in $\mathcal{R}'_n$.  We say $R \in \mathcal{R}'_n,$ if
	\begin{enumerate}[(i)]
		\item ${\ESp}_{2n}(R) = {\Sp}_{2n}(R)$ and
		
		\item If there exist finite rings  $R=R_1, R_2 \ldots, R_k=R'$ for $k>1,$ such that for $1 < i \leq k$,
		
		$R_i$ is either $R_{i-1}(X)$ or ${(R_{i-1})}_{\m}$ where $\m \in \Max(R_{i-1})$
		Then  ${\ESp}_{2n}(R') = {\Sp}_{2n}(R')$.
	\end{enumerate}
	
	\noindent Then we have the following interesting result:
	
	\begin{proposition}\label{l7}
		Let $R \in \mathcal{R}'_n$ and $s, t \geq 0$. Then 
		$${\Sp}_{2n}(R[X_1, \ldots, X_t, Y_1^{\pm{1}}, \ldots, Y_s^{\pm{1}}]) = {{\ESp}_{2n}}(R[X_1, \ldots, X_t, Y_1^{\pm{1}}, \ldots, Y_s^{\pm{1}}]).$$
	\end{proposition}
	
	\begin{proof}
		Let $s=0$. Then we need to show the equality  
		$${\Sp}_{2n}(R[X_1, \ldots, X_t]) = {{\ESp}_{2n}}(R[X_1, \ldots, X_t]).$$ 
		Let $\beta \in {\Sp}_{2n}(R[X_1, \ldots, X_t])$ and $R' = R(X_1, \ldots, X_{t-1})$. Since $R' \in \mathcal{R}_n,$ $\beta \in \Ep_{2n}(R'(X_t)) \cap {\Sp}_{2n}(R'[X_t])$. There exists a monic polynomial $f \in R'[X_t],$ such that $\beta_f \in \Ep_{2n}(R'[X_t]_f)$. By Theorem 3.10 of \cite{Kop-MR497932}, $\beta \in \Ep_{2n}(R'[X_t])$.

		Let $R''=R[X_1, \ldots, X_t, Y_1^{\pm{1}}, \ldots, Y_{s-1}^{\pm{1}}]$ and $S_1 (~S_2, \text{ resp.})$ be the multiplicative subset of $R''[Y_s, Y_s^{-1}]$, consisting of all monics in $Y_s (~Y_s^{-1}, \text{ resp.})$. For the case of Laurent polynomial rings, observe that for $\alpha \in {\Sp}_{2n}(R[X_1, \ldots, X_t, Y_1^{\pm{1}}, \ldots, Y_s^{\pm{1}}])$ we have
		$$\alpha \in {\Sp}_{2n}(R(Y_s)[X_1, \ldots, X_t, Y_1^{\pm{1}}, \ldots, Y_{s-1}^{\pm{1}}]. $$
		Since $R(Y_s) \in \mathcal{R}_n,$ thus by induction on $s+t$ we may conclude 
		$$\alpha \in {\ESp}_{2n}(R(Y_s)[X_1, \ldots, X_t, Y_1^{\pm{1}}, \ldots, Y_{s-1}^{\pm{1}}]) \subseteq {\ESp}_{2n}(S_1^{-1}R''[Y_s, Y_s^{-1}]).$$
		Similarly we may show $\alpha \in {\ESp}_{2n}(R(Y_s^{-1})[X_1, \ldots, X_t, Y_1^{\pm{1}}, \ldots, Y_{s-1}^{\pm{1}}]) \subseteq {\ESp}_{2n}(S_2^{-1}R''[Y_s, Y_s^{-1}]).$
		By Lemma 4.1 of \cite{Kop-MR497932}, we may conclude and say $\alpha \in {\ESp}_{2n}(R[X_1, \ldots, X_t, Y_1^{\pm{1}}, \ldots, Y_{s-1}^{\pm{1}}][Y_s, Y_s^{-1}]).$
	\end{proof}
	
	\noindent Using the above proposition, the reader may show that Theorem \ref{t1} holds for rings in $\mathcal{R}'_n$ as well.
	
	\begin{theorem}\label{t5}
		Let  $R \in \mathcal{R}'_n$ for some $n$ and $L$ be a c-divisible monoid, where $c > 1$. Then $${\ESp}_{2n}(R[L]) = {\Sp}_{2n}(R[L]).$$ 
	\end{theorem}
	
	We may similarly obtain the below from Theorem \ref{t5}:

	\begin{theorem}
		Let $R \in \mathcal{R}'_t$ and $L$ be a $c$-divisible monoid. Let $P$ be a symplectic $R[L]$-module extended from $R$. If $rank(P)=2t-2,$ then
		$${\Sp}(P \perp R[L]^2) =  {\Etrans}_{\Sp}(P \perp R[L]^2).$$
	\end{theorem}
	

	\section{Unstable ${\KK}_1$ for c-divisible monoids}
	
	We discuss the unstable ${\KK}_1$ results that we obtained for the linear case, building on the work in \cite{Gubeladze-ClassicalMR1079964}:
	\newline
	
	In his paper \cite{Gubeladze-ClassicalMR1079964}, Gubeladze showed that if $R$ is a local PID and $n \geq 3$, then for any $c$-divisible monoid $L,$ 
	$${\Sp}_n(R[L]) = {\El_{n}(R[L])}.$$
	\noindent He established this by developing class of rings $\mathbb{E}_{n}$ (much like the $\mathcal{R}_k$ we defined) and showed that any local PID belongs to $\mathbb{E}_{n}$ for $n \geq 3$. We generalize this result and show that for regular rings for large enough values of $n,$ we can establish unstable ${\KK}_1$-results.
	
	\begin{remark}
		Let $R$ be a regular affine algebra over $\mathbb{Q}$. Then, by (\cite{RaoKallen-MR1317126}) we have injective stability of ${\KK}_1(R)$ at the stage $n = \max\{3,d+1\}$. Further, if $R$ is local, then using (\cite{Vorst-regularMR606650}, Theorem 3.3) we may show for $n \geq 3,$ 
		${\Sl}_{n}(R[X_1, \ldots, X_m, Y_1^{\pm{1}}, \ldots, Y_{t}^{\pm{1}}] = {\El}_{n}(R[X_1, \ldots, X_m, Y_1^{\pm{1}}, \ldots, Y_{t}^{\pm{1}}]),$ for all $m ,t \geq 0$.
	\end{remark}
	
	\begin{remark}
		Let $R \in \mathbb{E}_{n}$, for some $n$ where $n \geq \dim(R)+1.$ Then $R \in \mathbb{E}_{k}$ for all $k \geq n$. 
		
	\end{remark}
	
	Using the above remarks and techniques of Example \ref{r5} in conjunction with (\cite{Suslin-MR0472792}, Theorem 6.3), we can show the following:
	
	\begin{proposition}
		Let $R$ be a regular local ring of dimension $d$. Then $R \in \mathbb{E}_{n},$ if
		\begin{enumerate}[(i)]
			\item  $n \geq \max\{3,d+2\}$,
			\item If $R$ is essentially of finite type over $\mathbb{Q}$ and $n \geq \max\{3,d+1\}$.
		\end{enumerate}
	\end{proposition}
	
	Then we can prove the following using (\cite{Suslin-MR0472792}) and emulate the techniques we employed in proof of Theorem \ref{t6}:
	
	\begin{theorem}
		Let $R$ be a regular ring of dimension d with ${\KK}_1(R)=0$ and $L$ be a $c$-divisible monoid. Then for any projective $P$ of rank $n$ with $n \geq \max\{2,d+1\},$
		$${\Sll}(P \oplus R[L]) =  {\ETranss}(P \oplus R[L]).$$
		Further, if  $R$ is essentially of finite type over $\mathbb{Q}$, then for $n \geq \max\{2,d\},$ we get ${\Sll}(P \oplus R[L]) =  {\ETranss}(P \oplus R[L])$.
	\end{theorem}
	
	\it{Acknowledgement:} The second author would like to thank Prof. Manoj Keshari for his valuable help in clarifying their doubts.

	\bibliographystyle{abbrv}
	\bibliography{reference}
\end{document}